\documentclass[11pt]{amsart}
\usepackage{amsopn}
\usepackage{amssymb, amscd}
\usepackage{multirow}
\usepackage{graphicx, graphics, epsfig}
\usepackage{faktor} 
\usepackage{enumerate}
\usepackage{tikz}

\usepackage{hyperref}
\usepackage{color}

\newcommand{\nc}{\newcommand}

\nc{\fg}{\mathfrak{f} } \nc{\vg}{\mathfrak{v} } \nc{\wg}{\mathfrak{w} }
\nc{\zg}{\mathfrak{z} } \nc{\ngo}{\mathfrak{n} } \nc{\kg}{\mathfrak{k} }
\nc{\mg}{\mathfrak{m} } \nc{\bg}{\mathfrak{b} } \nc{\ggo}{\mathfrak{g} }
\nc{\ggob}{\overline{\mathfrak{g}} } \nc{\sog}{\mathfrak{so} }
\nc{\sug}{\mathfrak{su} } \nc{\spg}{\mathfrak{sp} } \nc{\slg}{\mathfrak{sl} }
\nc{\glg}{\mathfrak{gl} } \nc{\cg}{\mathfrak{c} } \nc{\rg}{\mathfrak{r} }
\nc{\hg}{\mathfrak{h} } \nc{\tg}{\mathfrak{t} } \nc{\ug}{\mathfrak{u} }
\nc{\dg}{\mathfrak{d} } \nc{\ag}{\mathfrak{a} } \nc{\pg}{\mathfrak{p} }
\nc{\sg}{\mathfrak{s} } \nc{\affg}{\mathfrak{aff} } \nc{\qg}{\mathfrak{q} }

\nc{\pca}{\mathcal{P}} \nc{\nca}{\mathcal{N}} \nc{\lca}{\mathcal{L}}
\nc{\oca}{\mathcal{O}} \nc{\mca}{\mathcal{M}} \nc{\tca}{\mathcal{T}}
\nc{\aca}{\mathcal{A}} \nc{\cca}{\mathcal{C}} \nc{\gca}{\mathcal{G}}
\nc{\sca}{\mathcal{S}} \nc{\hca}{\mathcal{H}} \nc{\bca}{\mathcal{B}}
\nc{\dca}{\mathcal{D}} \nc{\wca}{\mathcal{W}} 

\nc{\val}{\operatorname{val}}
\nc{\vp}{\varphi} \nc{\ddt}{\tfrac{d}{dt}} \nc{\dds}{\frac{d}{ds}}
\nc{\dpar}{\tfrac{\partial}{\partial t}} \nc{\im}{\mathrm{i}}

\nc{\SO}{\mathrm{SO}} \nc{\Spe}{\mathrm{Sp}} \nc{\Sl}{\mathrm{SL}}
\nc{\SU}{\mathrm{SU}} \nc{\Or}{\mathrm{O}} \nc{\U}{\mathrm{U}} \nc{\Gl}{\mathrm{GL}}
\nc{\Se}{\mathrm{S}} \nc{\Cl}{\mathrm{Cl}} \nc{\Spein}{\mathrm{Spin}}
\nc{\Pin}{\mathrm{Pin}} \nc{\G}{\mathrm{GL}_n} \nc{\g}{\mathfrak{gl}_n}

\nc{\RR}{{\Bbb R}} \nc{\HH}{{\Bbb H}} \nc{\CC}{{\Bbb C}} \nc{\ZZ}{{\Bbb Z}}
\nc{\FF}{{\Bbb F}} \nc{\NN}{{\Bbb N}} \nc{\QQ}{{\Bbb Q}} \nc{\PP}{{\Bbb P}} \nc{\OO}{{\Bbb O}}

\nc{\vs}{\vspace{.2cm}} \nc{\vsp}{\vspace{1cm}} \nc{\ip}{\langle\cdot,\cdot\rangle}
\nc{\ipp}{(\cdot,\cdot)} \nc{\la}{\langle} \nc{\ra}{\rangle} \nc{\unm}{\tfrac{1}{2}} \nc{\unt}{\tfrac{1}{3}}
\nc{\unc}{\tfrac{1}{4}} \nc{\uns}{\tfrac{1}{6}} \nc{\no}{ \noindent}
\nc{\lam}{\Lambda^2(\RR^n)^*\otimes\RR^n} \nc{\tangz}{{\rm T}^{\rm Zar}}
\nc{\nor}{{\sf n}}  \nc{\mum}{/\!\!/} \nc{\kir}{/\!\!/\!\!/}
\nc{\Ri}{\tfrac{4\Ric_{\mu}}{||\mu||^2}} \nc{\ds}{\displaystyle}
\nc{\ben}{\begin{enumerate}} \nc{\een}{\end{enumerate}} \nc{\f}{\frac}
\nc{\lb}{[\cdot,\cdot]} \nc{\isn}{\tfrac{1}{||v||^2}}
\nc{\gkp}{(\ggo=\kg\oplus\pg,\ip)} \nc{\ukh}{(\ug=\kg\oplus\hg,\ip)}
\nc{\tgkp}{(\tilde{\ggo}=\kg\oplus\pg,\ip)}
\nc{\wt}{\widetilde} \nc{\Mm}{M}
\nc{\iop}{\mathtt{i}} \nc{\jop}{\mathtt{j}}

\nc{\Hess}{\operatorname{Hess}} \nc{\ad}{\operatorname{ad}}
\nc{\Ad}{\operatorname{Ad}} \nc{\rank}{\operatorname{rank}}
\nc{\Irr}{\operatorname{Irr}} \nc{\End}{\operatorname{End}}
\nc{\Aut}{\operatorname{Aut}} \nc{\Inn}{\operatorname{Inn}}
\nc{\Der}{\operatorname{Der}} \nc{\Ker}{\operatorname{Ker}}
\nc{\Iso}{\operatorname{Iso}} \nc{\Diff}{\operatorname{Diff}}
\nc{\Lie}{\operatorname{L}} \nc{\tr}{\operatorname{tr}} \nc{\dif}{\operatorname{d}}
\nc{\sen}{\operatorname{sen}} \nc{\modu}{\operatorname{mod}}
\nc{\CRic}{\operatorname{PP}} \nc{\Cric}{\operatorname{P}} \nc{\Ricci}{\operatorname{Ric}}
\nc{\sym}{\operatorname{sym}} \nc{\herm}{\operatorname{herm}} \nc{\symac}{\operatorname{sym^{ac}}}
\nc{\symc}{\operatorname{sym^{c}}} \nc{\scalar}{\operatorname{scal}}
\nc{\grad}{\operatorname{grad}} \nc{\ricci}{\operatorname{Rc}}
\nc{\Nor}{\operatorname{Norm}}  \nc{\ricc}{\operatorname{Rc^{c}}}
\nc{\Ricc}{\operatorname{Ric^{c}}} \nc{\ricac}{\operatorname{Rc^{ac}}}
\nc{\Ricac}{\operatorname{Ric^{ac}}} \nc{\Riem}{\operatorname{Rm}} \nc{\Sec}{\operatorname{Sec}}
\nc{\riccig}{\operatorname{ric^{\gamma}}} \nc{\Rin}{\operatorname{M}}
\nc{\Le}{\operatorname{L}} \nc{\tang}{\operatorname{T}}
\nc{\level}{\operatorname{level}} \nc{\rad}{\operatorname{r}}
\nc{\abel}{\operatorname{ab}} \nc{\CH}{\operatorname{CH}} \nc{\Cone}{{\mathcal C}} \nc{\CCone}{\operatorname{CC}} \nc{\CP}{{\mathcal P}}
\nc{\mcc}{\operatorname{mcc}} \nc{\Adj}{\operatorname{Adj}}
\nc{\Order}{\operatorname{O}}  \nc{\inj}{\operatorname{inj}} \nc{\proy}{\operatorname{pr}}
\nc{\vol}{\operatorname{vol}} \nc{\Diag}{\operatorname{Diag}} \nc{\Diagg}{\operatorname{Diag}}
\nc{\Spec}{\operatorname{Spec}} \nc{\Ima}{\operatorname{Im}} \nc{\Rea}{\operatorname{Re}}
\nc{\spann}{\operatorname{span}} \nc{\Aff}{\operatorname{Aff}} \nc{\mm}{\operatorname{m}} 
\nc{\Crit}{\operatorname{Crit}} \nc{\En}{\operatorname{E}}

\theoremstyle{plain}
\newtheorem{theorem}{Theorem}[section]
\newtheorem{proposition}[theorem]{Proposition}
\newtheorem{corollary}[theorem]{Corollary}
\newtheorem{lemma}[theorem]{Lemma}

\theoremstyle{definition}
\newtheorem{definition}[theorem]{Definition}

\theoremstyle{remark}
\newtheorem{remark}[theorem]{Remark}

\newtheorem{example}[theorem]{Example}

\title{A new example of a compact ERP $G_2$-structure}

\author{Ines Kath} \author{Jorge Lauret} 

\address{Institut f\"ur Mathematik und Informatik, Universit\"at Greifswald, Germany}
\email{ines.kath@uni-greifswald.de}

\address{Universidad Nacional de C\'ordoba and CIEM, CONICET, Argentina}
\email{jorgelauret@unc.edu.ar}

\thanks{The second author gratefully acknowledges support from Univ. Nac. de C\'ordoba, Argentina}

\begin{document}

\maketitle

\begin{abstract}
We provide the second known example of an extremally Ricci pinched closed $G_2$-structure on a compact $7$-manifold, by finding a lattice in the only unimodular solvable Lie group admitting a left-invariant $G_2$-structure.  Furthermore, the Laplacian coflow and its solitons are studied on a $6$-parameter family of left-invariant coclosed $G_2$-structures on this Lie group.  In this way, we obtain a $4$-parameter subfamily of expanding solitons.  The family is locally pairwise non-equivalent.  
\end{abstract}

\tableofcontents

\section{Introduction}\label{intro2}

A $G_2$-{\it structure} on a $7$-dimensional differentiable manifold $M$ can be identified with a positive (or definite) differential $3$-form $\vp$ on $M$, which naturally determines a Riemannian metric $g$ on $M$ and an orientation (see Section \ref{geom}).  Closed $G_2$-structures (i.e., $d\vp=0$) are very natural candidates to be deformed via the Laplacian flow $\dpar\vp(t) = \Delta\vp(t)$ toward {\it torsion-free} $G_2$-structures (i.e., $d\vp=0$ and $d\ast\vp=0$), which are the ones producing Ricci flat Riemannian metrics with holonomy contained in $G_2$ (see the survey \cite{Lty} for an account of recent advances).  On the other hand, closed $G_2$-structures are intended to play the same role as almost-K\"ahler structures in almost-hermitian geometry, so they have been extensively studied from different points of view (see e.g.\ \cite{Bry,ClyIvn1, ClyIvn2,FrnFinRff,LtyWei,Fields,PdsRff}). 

In the search for distinguished closed $G_2$-structures on a given compact manifold $M$, Bryant discovered the following remarkable curvature estimate (see \cite[Corollary 3]{Bry}):
\begin{equation}\label{estB}
\int_M \scalar^2 \ast 1 \leq 3\int_M |\Ricci|^2 \ast 1,
\end{equation}
where $\scalar$ and $\Ricci$ are respectively the scalar and Ricci curvature of $(M,g)$, and called a closed $G_2$-structure {\it extremally Ricci-pinched} (ERP for short) when equality holds in \eqref{estB} (see \cite[Remark 13]{Bry}).  Thus ERP structures are in some sense the closest you can get to Einstein among closed $G_2$-structures.  Bryant also proved that ERP $G_2$-structures are characterized by,
\begin{equation}\label{ERPeq}
d\tau = \tfrac{1}{6}|\tau|^2\vp + \tfrac{1}{6}\ast(\tau\wedge\tau),
\end{equation}
where $\tau:=-\ast d\ast\vp$ is the torsion $2$-form of $\vp$, and that in the compact case, this is indeed the only way in which $d\tau$ can quadratically depend on $\tau$ (see \cite[(4.66)]{Bry}).  A closed $G_2$-structure on any $7$-manifold (not necessarily compact) such that \eqref{ERPeq} holds is also called ERP.   

The ERP condition turned out to be really strong.  A first example was provided by Bryant in \cite{Bry} on the homogeneous space $\Sl_2(\CC)\ltimes\CC^2/\SU(2)$.  More recently, the classification list of all left-invariant ERP $G_2$-structures on Lie groups up to equivalence and scaling was obtained in \cite{ERP,ERP2}, which consists of only five simply connected completely solvable Lie groups (in particular all diffeomorphic to $\RR^7$), each one admitting a single structure (one of them is equivalent to Bryant's example).  These five structures are also {\it steady Laplacian solitons}:  the Laplacian flow solution starting at $\vp$ is simply given by $\vp(t)=f(t)^*\vp$, for some $f(t)\in\Diff(M)$, $t\in\RR$ (the set of ERP $G_2$-structures is known to be invariant under the Laplacian flow and the solutions are always eternal, see  \cite{FinRff2}). 

The only other known ERP $G_2$-structures are many non-homogeneous examples found in \cite{Bll} and a continuous family of left-invariant structures on Lie groups given in \cite{FinRff2}, which are all equivalent to Bryant's example but the Lie groups involved are pairwise non-isomorphic (these are precisely the structures to be added in order to obtain a classification up to equivariant equivalence, see \cite{ERP2}).   

Remarkably, the only known compact example in the literature so far is the quotient of Bryant's homogeneous example $\Sl_2(\CC)\ltimes\CC^2/\SU(2)$ by a cocompact discrete subgroup of $\Sl_2(\CC)\ltimes\CC^2$, giving rise to a locally homogeneous manifold.  It was conjectured by Cleyton and Ivanov in the Introduction of \cite{ClyIvn2} that this is the only compact ERP structure up to local equivalence.  The main result of this paper is to disprove the conjecture by providing a new example of an ERP $G_2$-structure on a compact manifold.  We construct a {\it lattice} (i.e., a discrete and cocompact subgroup) in one of the five solvable Lie groups appearing in the classification list given in \cite{ERP2}, which will be denoted by $G_J$ from now on.  Such example was originally given in \cite[Example 4.7]{LS-ERP} and it is the only one in the list which is unimodular, so the only one with possibilities to admit a lattice.  The alternative presentation of Bryant's example as a left-invariant $G_2$-structure on a solvable Lie group $S$ given in \cite[Remark 6.7]{ClyIvn2} (see also {\cite[Examples 4.13, 4.10]{LS-ERP}), allows us to deduce that it is not equivalent to our example on $G_J$, as $S$ and $G_J$ are non-isomorphic simply connected completely solvable Lie groups (see \cite{Alk} or Section \ref{equiv0}).  In particular, the corresponding compact quotients cannot be locally equivalent.  

The Lie algebra of $G_J$ is given by $\ggo_J=\ag\ltimes\RR^4$, where $\ag\subset\slg_4(\RR)$ is the subspace of all diagonal matrices.  Note that $\ggo_J$ is isomorphic to any semidirect product of the form $\ag\ltimes\ngo$, where $\dim{\ngo}=4$, $[\ngo,\ngo]=0$ and $\ag$ is any maximal $\RR$-split torus of $\slg(\ngo)$ (cf.\ \cite[Example 5.4]{ERP}).  The lattices $\Gamma\subset G_J$ we found are of the form 
$$
\Gamma:=\exp(\ZZ A+\ZZ B+\ZZ C)\ltimes\phi(\ZZ^4),
$$  
where $\{ A,B,C\}$ is a basis of $\ag$ such that $e^A$, $e^B$ and $e^C$ leave invariant the lattice $\phi(\ZZ^4)$ of $\RR^4$ for a suitable $\phi\in\Gl_4(\RR)$. The maps $e^A$, $e^B$ and $e^C$ and the automorphism $\phi$ are constructed as follows. We take a totally real quartic number field $K$ and consider the ring $\oca_K$ of integers. Clearly, each unit in $\oca_K$ acts by multiplication on $\oca_K$. With respect to a basis, this action is given by an integer matrix, and we can choose three multiplicatively independent units. Since the units commute, also the corresponding integer matrices commute and so they are simultaneously diagonalisable over $\RR$ by a matrix $\phi$. The diagonal matrices that we obtain in this way can be used as $e^A$, $e^B$ and $e^C$. Since the integer matrices we found leave invariant $\ZZ^4$, the corresponding diagonalised matrices leave invariant $\phi(\ZZ^4)$. The multiplicatively independency of the chosen units will ensure, that $\Gamma$ is indeed a lattice. This will be explained in more detail in Section~\ref{ant}. A first explicit example can be found in Section~\ref{latt-sec}.

In Section \ref{geom}, we study a $9$-parameter family of left-invariant $G_2$-structures on the Lie group $G_J$.  A very particular feature is that each structure in the family is rigid, in the sense that it has a pairwise non-equivalent open neighborhood.  Note that all these structures descend to the compact manifold $M=G_J/\Gamma$, where $\Gamma$ is any of the lattices we have found, to become locally homogeneous $G_2$-structures.  We use results in \cite{Ncl} to provide formulas for the torsion forms and Laplacians, obtaining that all these $G_2$-structures belong to the class $\wca_2\oplus\wca_3$ (i.e., $\tau_0=0$ and $\tau_1=0$), which from the spinorial viewpoint it is equivalent to have a harmonic associated unit spinor (see \cite{Agr}).  Finally, the Laplacian coflow and its solitons are studied on the $6$-parameter subfamily of coclosed $G_2$-structures.  We prove long-time existence among two different subfamilies and obtain a $4$-parameter family of expanding Laplacian coflow solitons.   

\vs \noindent {\it Acknowledgements.}  We thank Alberto Raffero for very helpful comments.

\section{Explicit example of a lattice}\label{latt-sec}

We exhibit in this section an explicit example of a lattice in the simply connected solvable Lie group $G_J$ with Lie algebra $\ggo_J=\ag\ltimes\RR^4$, where $\ag$ is the subspace of all diagonal matrices in $\slg_4(\RR)$. Note that $G_J=\exp(\ag)\ltimes\RR^4$.  

We consider the matrices
\begin{equation}\label{matr}
A_1:=\left[\begin{smallmatrix} 
0&0&-1&-1\\ 
0&0&-4&-5\\
1&0&4&0\\
0&1&1&5
\end{smallmatrix}\right], \qquad 
A_2:=\left[\begin{smallmatrix} 
3&-1&-1&-1\\ 
-4&-1&-5&-5\\
0&0&3&-1\\
1&1&1&4
\end{smallmatrix}\right], \qquad
A_3:=\left[\begin{smallmatrix} 
4&1&2&3\\ 
3&8&9&14\\
-1&-1&0&-3\\
-1&-2&-3&-3
\end{smallmatrix}\right],
\end{equation}
which have determinant one and so they all belong to $\Sl_4(\ZZ)$.  They also have the same spectrum, consisting of four different positive real numbers $\{u_1^2,\dots,u_4^2\}$, where 
\begin{equation}\label{u1234}
u_1:=2\cos\left(2\pi \tfrac{1}{15}\right), \quad u_2:=2\cos\left(2\pi \tfrac{2}{15}\right), \quad u_3:=2\cos\left(2\pi \tfrac{4}{15}\right), \quad u_4:=2\cos\left(2\pi \tfrac{7}{15}\right).
\end{equation}
Moreover, it is straightforward to check that they pairwise commute and simultaneously diagonalize as follows: 
\begin{eqnarray*}
&\phi A_1\phi^{-1}= \Diag(u_1^2,u_2^2,u_3^2,u_4^2 ), \quad
\phi A_2\phi^{-1}=\Diag (u_2^2,u_3^2,u_4^2,u_1^2), &\\
&\phi A_3\phi^{-1}=\Diag(u_3^2,u_4^2,u_1^2,u_2^2), &
\end{eqnarray*} 
where $\phi$ equals the Vandermonde matrix $V(u_1,u_2,u_3,u_4)$, that is,
$$
\phi=\left[ \begin{smallmatrix}  1&u_1&u_1^2&u_1^3\\ 1&u_2&u_2^2&u_2^3\\1&u_3&u_3^2&u_3^3\\1&u_4&u_4^2&u_4^3 \end{smallmatrix}\right].
$$
The linear maps $\phi A_j\phi^{-1}$, $j=1,2,3$, generate a lattice $\Lambda$ of $\exp(\ag)$.  
Since they leave invariant the subset $\phi(\ZZ^4)\subset\RR^4$, the set
$
\Gamma:=\Lambda\ltimes\phi(\ZZ^4)
$ 
is a subgroup of $G_J$.  Moreover, $\Gamma$ is a lattice in $G_J$ since $\Lambda\subset\exp(\ag)$ and $\phi(\ZZ^4)\subset\RR^4$ are both discrete and cocompact.   

It was proved in \cite[Example 4.7]{LS-ERP} that if we call $\{ e_3,\dots,e_6\}$ the standard basis of $\RR^4$ and consider the basis of $\ag$ given by  
$$
e_7:=\Diag(1,1,-1,-1), \quad
e_1:=\Diag(1,-1,1,-1), \quad
e_2:=\Diag(1,-1,-1,1),   
$$
then the left-invariant $G_2$-structure on $G_J$ determined by the positive $3$-form 
$$
\vp=e^{127}+e^{347}+e^{567}+e^{135}-e^{146}-e^{236}-e^{245},
$$
is ERP.  Thus $\vp$ also defines an ERP $G_2$-structure on the compact manifold $M=G_J/\Gamma$ (see Example \ref{abc} for more information).

\section{Algebraic number theory and lattices}\label{ant}

In this section we want to explain the idea behind the lattice described in Section~\ref{latt-sec} and a way to get more lattices. 

Let $T$ be the identity component of an $\RR$-split torus in $\Sl_4(\RR)$ and define $G_T:=T\ltimes \RR^4$. In this notation, $G_J=G_{T^0}$, where $T^0:=\exp(\ag)$. Then $G_T$ is isomorphic to $G_J$ for each such $T$. Indeed, choose $\phi\in \Gl_4(\RR)$ such that $\phi T \phi^{-1}=T^0$. Then we can define an isomorphism by 
\begin{equation}\label{phi}
G_T \longrightarrow G_J,\quad (A,b) \longmapsto (\phi A \phi^{-1}, \phi(b)). 
\end{equation}
In the following, we will construct several such tori and for each torus $T$ a lattice in $G_T$. Of course, each such lattice is mapped to a lattice in $G_J$ by the isomorphism defined in~\eqref{phi}.

Let us take an arbitrary totally real quartic number field $K:=\QQ(u)$. Let $$p(t)=t^4+a_3t^3+\dots+a_0\in\QQ[t]$$ be the minimal polynomial of $u$. We consider $K$ as a vector space over $\QQ$.  Multiplication by $u$ is an invertible linear map on $K$. Let $M\in\Gl_4(\QQ)$ be the matrix of this map with respect to the basis $1, u, u^2, u^3$ of the vector space $K$ (over $\QQ$). Then $M$ is the companion matrix of $p$, that is,
\begin{equation}\label{comp}
 M=\left[\begin{smallmatrix} 
0&0&0&-a_0 \\
1&0&0&-a_1 \\
0&1&0&-a_2 \\
0&0&1&-a_3 \end{smallmatrix}
\right],
\end{equation}
and the eigenvalues of $M$ are the roots of $p$. In particular, $M$ is diagonalisable and all its eigenvalues are real. The ring $\oca_K$ of integers in $K$ is a free $\ZZ$-module of rank four. Since $K$ is totally real, the group of units in $\oca_K$ has rank three by Dirichlet's unit theorem. Hence one can choose three multiplicatively independent units. These units act as commuting automorphisms on $\oca_K$. In particular, these actions are given by integer matrices $A_1,A_2,A_3\in \Gl_4(\ZZ)$ with respect to an integral basis. The eigenvalues of these matrices are real, since the units are linear combinations of $1, u, u^2, u^3$, thus the corresponding matrices are linear combinations of $I_4,M, M^2, M^3$ and these matrices are simultaneously diagonalisable with real eigenvalues. In particular, $\tilde T:=\spann_\RR\{I_4,M, M^2, M^3\}\cap \Gl_4(\RR)$ is an $\RR$-split torus. Let $T$ be its identity component. The linear maps $A_1,A_2,A_3$ generate a lattice in $\tilde T$, and if $\Lambda$ is its intersection with $T$, then $\Gamma:=\Lambda\ltimes \ZZ^4$ is a lattice in $G_T$.

\begin{definition}
Two lattices $\Gamma_1$ and $\Gamma_2$ in a Lie group $G$ are called commensurable if there exist finite index subgroups $\Gamma_1'\subset\Gamma_1$, $\Gamma_2'\subset\Gamma_2$ and an automorphism $\Phi$ of $G$ such that $\Phi(\Gamma_1')=\Gamma_2'$.
\end{definition}

\begin{lemma}\label{comm}
Let $\Gamma_j=\Lambda_j\ltimes \ZZ^4\subset G_{T_j}$, $j=1,2$, be lattices as constructed above starting from totally real quartic fields $\QQ(u_j)$ and let $p_j$ be the minimal polynomial of $u_j$. If $\Gamma_1\subset G_{T_1}\cong G_J$ and $\Gamma_2 \subset G_{T_2}\cong G_J$ are commensurable, then the splitting fields of $p_1$ and $p_2$ are isomorphic.
\end{lemma}
\begin{proof}  Suppose that the images of $\Gamma_1$ and $\Gamma_2$ in $G_J$ are commensurable. By definition, there are finite index subgroups
$\Gamma_1'\subset\Gamma_1$, $\Gamma_2'\subset\Gamma_2$ and an isomorphism $\Phi: G_{T_1}\rightarrow G_{T_2}$ such that $\Phi(\Gamma_1')=\Gamma_2'$.

Since $\Phi: G_{T_1}\rightarrow G_{T_2}$ is an isomorphism, $\Phi$ maps $\RR^4$ to $\RR^4$, thus it is of the form
\begin{eqnarray*}
\Phi: G_{T_1}=T_1\ltimes \RR^4 & \longrightarrow &G_{T_2}=T_2\ltimes \RR^4\\
(A,0)&\longmapsto & (\phi_1(A),\phi_2(A))\\
(I_4,b)&\longmapsto &(I_A, \phi(b)). 
\end{eqnarray*}
Moreover, $\phi_1(A)=\phi A \phi^{-1}.$ In particular, we have $T_2=\phi T_1 \phi^{-1}.$

Since $\Phi$ maps $\Gamma_1'$ to $\Gamma_2'$, the restriction $\Phi|_{\RR^4}=\phi:\RR^4\rightarrow\RR^4$ maps the finite index subgroup $\Gamma_1'\cap \RR^4$ of $\ZZ^4$ to $\Gamma_2'\cap\RR^4\subset \ZZ^4$. Consequently, $\phi$ belongs to $\Gl_4(\QQ)$. We consider the tori $(T_j)_{\QQ}:=T_j\cap \Sl_4(\QQ)$, $j=1,2$. Since $\phi$ has rational entries, $(T_2)_\QQ=\phi (T_1)_\QQ \phi^{-1}$ holds. The splitting field of $(T_j)_{\QQ}$ is by definition the smallest field extension $L_j$ of $\QQ$ such that $(T_j)_{\QQ}$ splits over $L_j$. Since $(T_1)_{\QQ}$ and $(T_2)_{\QQ}$ are conjugate in $\Gl_4(\QQ)$, their splitting fields must be isomorphic. 
Let $M_j\in \Gl_4(\QQ)$ denote the linear maps corresponding to $u_j$, $j=1,2$.  Since $$(T_j)_\QQ=T_j\cap\,\spann_\QQ\{I_4,M_j, M_j^2, M_j^3\}$$ and the eigenvalues of $M_j$ are the roots of the minimal polynomial $p_j\in\QQ[t]$ of $u_j$, the splitting field $L_j$ of $(T_j)_\QQ$ equals the splitting field of $p_j$. Thus the splitting fields of $p_1$ and $p_2$ are isomorphic.
 \end{proof}

\begin{example} \label{ex1}
We now explain how the example given in Section \ref{latt-sec} arises from the construction explained above. We consider the irreducible polynomial $p(t):=t^4-t^3-4t^2+4t+1\in\QQ[t]$, whose roots are the numbers
$u_1, u_2, u_3, u_4$ that were defined by~\eqref{u1234}.
Then $u_2=u_1^2-2$, $u_3=u_2^2-2$, $u_4=(u_1 u_2 u_3)^{-1}$. In particular, all roots belong to $K:=\QQ(u_1)$. Since $u_1,u_2,u_3,u_4$ are units in $\oca_K$, they act as commuting automorphisms on $\oca_K$ . The units $u_1,u_2,u_3$ are multiplicatively independent, see below. Since we want to construct automorphisms with positive eigenvalues, we choose the units $u_1^2, u_2^2, u_3^2$, which of course are also independent.  The matrices $A_1, A_2, A_3$ in \eqref{matr} are exactly the matrices corresponding to multiplication by  these units with respect to the integral basis $1, u_1, u_1^2, u_1^3$.   They can easily be computed using that the multiplication by $u_1$ is given by the companion matrix of $p$, see \eqref{comp}, and using the identities $u_2=u_1^2-2$ and $u_3=u_2^2-2$.
Let us denote the matrices corresponding to $u_1$, $u_2$ and $u_3$ by $B_1,B_2$ and $B_3$, respectively. The eigenvalues of each $B_j$, $j=1,2,3$, are $u_1,\dots,u_4$. Indeed, by construction $p(B_j)=0$, thus the minimal polynomial of $B_j$ divides $p$. Since $p$ is irreducible over $\QQ$, this implies that $p$ is the characteristic polynomial of $B_j$. This shows that the spectrum of each $A_j$, $j=1,2,3$ equals $\{u_1^2,\dots,  u_4^2\}$ as claimed above. The matrix $B_1$ is the companion matrix of $p$, thus it is diagonalised by the Vandermonde matrix $\phi=V(u_1,u_2,u_3,u_4)$. Since $A_1, A_2$ and $A_3$ are polynomials in $B_1$, they can all be simultaneously diagonalised by $\phi$.

We remark that $\QQ(u_1)$ is the splitting field of $p$. The Galois group of $p$ is cyclic of order four and generated by a map that acts on the roots by $u_1\mapsto u_2$, $u_2\mapsto u_3$, $u_3\mapsto u_4$, $u_4\mapsto u_1$. 

\noindent {\it Proof of independency.} We want to show that the units $u_1,u_2$ and $u_3$ are indeed multiplicatively independent. Assume that $u_1^ku_2^lu_3^m=1$ for $k,l,m\in\ZZ$. Using the action of the Galois group and the identity $u_1u_2u_3u_4=1$, we obtain the following system of equations
$$u_1^ku_2^lu_3^m=1, \quad u_1^{-m}u_2^{k-m}u_3^{l-m}=1,\quad u_1^{m-l}u_2^{-l}u_3^{k-l}=1.$$
The first and the third equation give $(u_1u_3)^{k+m-l}=1$, thus $l=k+m$. Inserting this into the first two equations, this yields
$$(u_1u_2)^k (u_2u_3)^m=1,\quad (u_1u_2)^{-m}(u_2u_3)^k=1  .$$
Combining these equations, we obtain $(u_1u_2)^{k^2+m^2}=1$, which gives $k=m=0$, thus also $l=0$.
\qed
\end{example}

In what follows, by also applying the above technique, we give a new example of a lattice of the solvable Lie group $G_J=\exp(\ag)\ltimes\RR^4$ considered in Section \ref{latt-sec}, which is not commensurable to Example \ref{ex1}.  

\begin{example}\label{ex2}
The matrices
\begin{equation}\label{matr}
A_1:=\left[\begin{smallmatrix} 
0&0&-1&0\\ 
0&0&0&-1\\
1&0&4&0\\
0&1&0&4
\end{smallmatrix}\right], \qquad 
A_2:=\left[\begin{smallmatrix} 
1&0&-4&-4\\ 
4&1&0&-4\\
4&4&17&16\\
0&4&4&17
\end{smallmatrix}\right], \qquad
A_3:=\left[\begin{smallmatrix} 
-5&-10&-20&-38\\ 
-2&-5&-10&-20\\
20&38&75&142\\
10&20&38&75
\end{smallmatrix}\right]
\end{equation}
have determinant one. We put 
\begin{equation}\label{ex2u}
u_1:=u:=\sqrt{2+\sqrt{3}}, \quad u_2=u^{-1}=\sqrt{2-\sqrt{3}}, \quad u_3=-u,\quad u_4=-u^{-1}.
\end{equation}
The Vandermonde matrix $\phi:=V(u_1,u_2,u_3,u_4)$ simultaneously diagonalises $A_1, A_2$ and $A_3$. More exactly,
$$
\phi A_1\phi^{-1}=M^2,\quad
\phi A_2\phi^{-1}=(1+2M)^2, \quad
\phi A_2\phi^{-1}= (1+M-2M^2-M^3)^2
$$
for $M:=\Diag(u_1,u_2,u_3,u_4).$
The linear maps $\phi A_j\phi^{-1}$, $j=1,2,3$, generate a lattice $\Lambda$ of $\exp(\ag)$, see below. Moreover, they leave invariant the subset $\phi(\ZZ^4)\subset\RR^4$.
Thus
$
\Gamma:=\Lambda\ltimes\phi(\ZZ^4)
$ 
is a lattice in $G_J$. 

The numbers $u_1,\dots,u_4$ defined in \eqref{ex2u} are the roots of $p(t)=t^4-4t^2+1$. We choose the units $(u)^2, (1+2u)^2$ and $(1+u-2u^2-u^3)^2$. These are multiplicatively independent, see below. The matrices $A_1,A_2, A_3$ correspond to the multiplication by these units  with respect to the integral basis $1, u,u^2,u^3$. Consequently, as in the previous example, they can all expressed as polynomials in the companion matrix of $p$. Thus they  can be simultaneously diagonalised by the Vandermonde matrix $\phi$.

Here, $\QQ(u)$ is the splitting field of $p$. Its Galois group of $p$ is isomorphic to the Klein four-group. It is generated by the map $\sigma_1$ that maps $u$ to $u^{-1}$ and the map $\sigma_2$ that maps $u$ to $-u$.

\noindent {\it Proof of independency.} We prove that the units $u, 1+2u$ and $1+u-2u^2-u^3$ are multiplicatively independent, thus also their squares are independent. Assume that 
\begin{equation}\label{E1}
u^k(1+2u)^l (1+u-2u^2-u^3)^m=1
\end{equation}
for some $k,l,m\in\ZZ$. Applying the elements $\sigma_1$ and $\sigma_2$ of the Galois group, we obtain
\begin{eqnarray}
&u^{-k}(1+2u^{-1})^l (1+u^{-1}-2u^{-2}-u^{-3})^m=1,& \label{E2}\\
&(-u)^k(1-2u)^l (1-u-2u^2+u^3)^m=1.&\label{E3}
\end{eqnarray}
Since $(1+u^{-1}-2u^{-2}-u^{-3})(1-u-2u^2+u^3)=-1$, the multiplication of the identities \eqref{E2} and \eqref{E3} gives
$$(-1)^{k+m}(1+2u^{-1})^l(1-2u)^l=(-1)^{k+m}(-3+6u-2u^3)^l =1,$$
hence $l=0$. Now we multiply the identities \eqref{E2} and \eqref{E3}, which gives
$$(1+u-2u^2-u^3)^m(1+u^{-1}-2u^{-2}-u^{-3})^m= (-5-10u+2u^3)^m=1.$$
This implies $m=0$, thus also $k=0$.
\qed
\end{example}

\begin{remark}
The lattices that we constructed in Example~\ref{ex1} and in Example~\ref{ex2} are not commensurable. This follows from Lemma \ref{comm} and the fact that the Galois groups of the used polynomials are not isomorphic as we have seen above. 
\end{remark}

\section{$G_2$-geometry on the group $G_J$}\label{geom}

We study in this section a large family of $G_2$-structures on the Lie group $G_J=\exp(\ag)\ltimes\RR^4$ considered in the above sections.  Recall that any left-invariant $G_2$-structure on $G_J$ also defines a $G_2$-structure on the compact $7$-manifold $M=G_J/\Gamma$ which is locally homogeneuos, where $\Gamma$ is any of the lattices of $G_J$ exhibited in Sections \ref{latt-sec} and \ref{ant}.  

A differential $3$-form $\vp$ on a $7$-dimensional differentiable manifold $M$ is called a $G_2$-{\it structure} when it is {\it positive} (or definite), in the sense that $\vp$ naturally determines a Riemannian metric $g$ on $M$ and an orientation by 
$$
g(X,Y)\vol = \frac{1}{6} \iota_X(\vp)\wedge \iota_Y(\vp)\wedge\vp, \qquad\forall X,Y\in\mathfrak{X}(M).
$$
The positivity of $\vp$ is equivalent to the fact that at each point $p\in M$, $\vp_p$ can be written as 
\begin{equation}\label{phi}
\vp_p=e^{127}+e^{347}+e^{567}+e^{135}-e^{146}-e^{236}-e^{245},
\end{equation}
with respect to some basis $\{ e_1,\dots,e_7\}$ of $T_pM$.

We start by recalling that the {\it torsion forms} of a $G_2$-structure $\vp$ on a manifold $M$ are the components of the {\it intrinsic torsion} $\nabla\vp$, where $\nabla$ is the Levi-Civita connection of the metric $g$ attached to $\vp$.  If we set 
$$
\psi:=\ast\vp,
$$ 
then they are defined as the unique differential forms $\tau_i\in\Omega^iM$, $i=0,1,2,3$, such that
\begin{equation}\label{torsion}
d\vp=\tau_0\psi+3\tau_1\wedge\vp+\ast\tau_3, \qquad d\psi=4\tau_1\wedge\psi+\tau_2\wedge\vp,
\end{equation}
and they are given by (see e.g., \cite[(4)]{MnrOtlVll}),
\begin{equation}\label{torsion2}
\begin{array}{c}
\tau_0 = \frac{1}{7}\ast(d\vp\wedge\vp), \qquad \tau_1 =  -\frac{1}{12}\ast(\ast d\vp\wedge\vp), \\  \\
\tau_2= -\ast d\psi+4\ast(\tau_1\wedge\psi), \qquad \tau_3 = \ast d\vp-\tau_0\vp-3\ast(\tau_1\wedge\vp).  
\end{array}
\end{equation}

Let $\ggo_{A,B,C}$ be the solvable Lie algebra with basis $\{ e_1,\dots,e_7\}$ such that $\ag:=\la e_7,e_1,e_2\ra$ is abelian, $\ngo:=\la e_3,e_4,e_5,e_6\ra$ is an abelian ideal and in terms of the basis $\{ e_3,\dots,e_6\}$, 
$$
\ad{e_7}|_\ngo=A, \quad
\ad{e_1}|_\ngo=B, \quad
\ad{e_2}|_\ngo=C, \qquad A,B,C\in\slg_4(\RR).  
$$
Note that the Jacobi condition holds if and only if $A,B,C$ pairwise commute.  Thus the corresponding simply connected solvable Lie group $G_{A,B,C}$ is isomorphic to our group $G_J$ if and only if $\{ A,B,C\}$ is linearly independent and simultaneously diagonalizes over $\RR$, that is, $\la A,B,C\ra$ is a maximal $\RR$-split torus of $\slg_4(\RR)$.  Such triples of matrices will be called {\it compatible}.  

We consider the left-invariant $G_2$-structure defined on each $G_{A,B,C}$ by the positive $3$-form $\vp$ given in \eqref{phi} (the basis $\{ e_1,\dots,e_7\}$ is orthonormal and oriented with respect to the inner product $\ip$ determined by $\vp$).  Each triple of commuting matrices is therefore identified with the $G_2$-structure $(G_{A,B,C},\vp)$, which in the compatible case is (equivariantly) equivalent to the left-invariant $G_2$-structure on $G_J$ defined by $h^*\vp$, where $h:G_J\rightarrow G_{A,B,C}$ is any Lie group isomorphism (see Section \ref{equiv0} below).  In this way, a large family of $G_2$-structures on $G_J$ is what one is really exploring by varying all compatible triples.  This point of view is often called the moving-bracket approach (see e.g., \cite{LF,sol-HS}).

\subsection{Equivalence}\label{equiv0}
Two $G_2$-structures $(M,\vp)$ and $(M',\vp')$ are said to be {\it equivalent} if there is a diffeomorphism $f:M\rightarrow M'$ such that $\vp=f^*\vp'$, and in the case when $M,M'$ are Lie groups and the structures are left-invariant, they are called {\it equivariantly equivalent} if in addition $f$ is a Lie group isomorphism.  Let us recall two key results concerning equivalence of left-invariant metrics.  

\begin{theorem}\label{equiv-cs} \cite[Theorem 1]{Alk} \cite[Theorem 5.2]{GrdWls}
If two simply connected completely solvable Lie groups endowed with left-invariant Riemannian metrics are isometric, then there exists an isomorphism between the Lie groups which is an isometry (i.e., they are equivariantly isometric).  
\end{theorem} 

\begin{theorem}\label{equiv-csu} \cite[Theorem 4.3]{GrdWls} 
Any isometry fixing the identity element of a simply connected completely solvable and unimodular Lie group endowed with a left-invariant Riemannian metric is an automorphism of the Lie group.  
\end{theorem}

Since any equivalence between $G_2$-structures is also an isometry between the corresponding attached metrics, it follows from Theorem \ref{equiv-cs} that two equivalent simply connected completely solvable Lie groups endowed with left-invariant $G_2$-structures $(G,\vp)$ and $(G',\vp')$ must be isomorphic.  Moreover, if the groups $G$ and $G'$ are in addition unimodular, then two equivalent $(G,\vp)$ and $(G',\vp')$ are automatically equivariantly equivalent by Theorem \ref{equiv-csu}.  Indeed, the equivalence $f:(G,\vp)\rightarrow (G',\vp')$ can be assumed to satisfy $f(e)=e'$ by left-invariance and if $h:G\rightarrow G'$ is the Lie group isomorphism provided by Theorem \ref{equiv-cs}, then the isometry $fh^{-1}$ of $(G',\ip')$ is necessarily an automorphism of $G'$ by Theorem \ref{equiv-csu}, from which it follows that $f:G\rightarrow G'$ is an isomorphism.         

In other words, the two notions of equivalence and equivariant equivalence coincide among the class of left-invariant $G_2$-structures on unimodular completely solvable Lie groups.  

Consider the $14$-dimensional compact simple Lie group 
$$
G_2:=\{ h\in\Gl_7(\RR):h^*\vp=\vp\}\subset\SO(7).  
$$
Any Lie group isomorphism $f:G\rightarrow G'$ such that, after identifying the underlying vector spaces of the respective Lie algebras, $h:=df|_e\in G_2$, therefore determines an equivalence between the corresponding left-invariant $G_2$-structures $(G,\vp)$ and $G',\vp)$. 

Since $G_J$ is completely solvable and unimodular, it follows from the results in \cite{Alk, GrdWls} described above that two compatible $(G_{A,B,C},\vp)$ and $(G_{A',B',C'},\vp)$ are equivalent if and only if there exists an isomorphism $h:\ggo_{A,B,C}\rightarrow\ggo_{A',B',C'}$ such that $h\in G_2$.  Being $\ngo$ the nilradical of both Lie algebras, we have that $h(\ngo)=\ngo$ and so $h(\ag)=\ag$ as well.  It is well known that $h_2:=h|_{\ngo}$ has determinant one and that actually for any $h_2\in\SO(4)$ there exists $h_1\in\SO(3)$ such that 
$$
h:=\left[\begin{matrix} h_1&0\\ 0&h_2 \end{matrix}\right]\in G_2.  
$$
(See e.g., \cite{VnLMnr}).  Thus two compatible $(G_{A,B,C},\vp)$ and $(G_{A',B',C'},\vp)$ are equivalent if and only if there is an $h_2\in\SO(4)$ such that 
\begin{align} 
A'&=x_{11}h_2Ah_2^{-1}+x_{21}h_2Bh_2^{-1}+x_{31}h_2Ch_2^{-1},\notag \\ 
B'&=x_{12}h_2Ah_2^{-1}+x_{22}h_2Bh_2^{-1}+x_{32}h_2Ch_2^{-1}, \qquad h_1^{-1}=[x_{ij}], \label{equiv1} \\
C'&=x_{13}h_2Ah_2^{-1}+x_{23}h_2Bh_2^{-1}+x_{33}h_2Ch_2^{-1} \notag
\end{align}
  
In this section, we will often use the homothety invariant for non-flat homogeneous metrics defined by 
\begin{equation}\label{RP}
F(g)= \frac{\scalar_g^2}{\tr{\Ricci_g^2}},
\end{equation}
to show that a given explicit one-parameter family of $G_2$-structures is pairwise non-homothetic.  Note that $F(g)\leq 7$ and equality holds if and only if $g$ is Einstein, hence $F$ measures in some sense how far is the metric $g$ from being Einstein (see \cite{Fields} for a study of the behavior of $F$ on homogeneous closed $G_2$-structures).

\subsection{Formulas for the torsion and the Laplacians}
In \cite{Ncl, Ncl-T}, a much larger class of Lie groups endowed with $G_2$-structures is studied, including computations of the torsion forms and the value of the Laplacian at $\vp$ and $\psi$.  In order to use such formulas, we introduce the following notation.  

We write the positive $3$-form given in \eqref{phi} as  
\begin{align}
\vp=\omega_7\wedge e^7+\omega_1\wedge e^1+\omega_2\wedge e^2+e^{127}, \label{phi2}
\end{align}
where 
$$
\omega_7:=e^{34}+e^{56}, \qquad \omega_1:=e^{35}-e^{46}, \qquad \omega_2:=-e^{36}-e^{45},
$$
and so 
\begin{align}
\psi:=\ast\vp=\omega_7\wedge e^{12}+\omega_1\wedge e^{27}-\omega_2\wedge e^{17}+e^{3456}. \label{sphi}
\end{align}
Let $\theta$ denote the natural representation of $\slg(\ngo)$ on $\Lambda^2\ngo^*$.  

\begin{proposition}\label{formulas}\cite{Ncl, Ncl-T}
For each $G_2$-structure $(G_{A,B,C},\vp)$, the following formulas hold:
\begin{align}
d\vp =& (\theta(B)\omega_7-\theta(A)\omega_1)\wedge e^{17} + (\theta(C)\omega_7-\theta(A)\omega_2)\wedge e^{27}  \label{dphi} \\ 
&+ (\theta(B)\omega_2-\theta(C)\omega_1)\wedge e^{12}, \notag\\ 
\ast d\vp=&(\theta(B^t)\omega_7-\theta(A^t)\omega_1)\wedge e^2-(\theta(C^t)\omega_7-\theta(A^t)\omega_2)\wedge e^1 \label{sdphi}\\
    &-(\theta(B^t)\omega_2-\theta(C^t)\omega_1)\wedge e^7, \notag \\
\ast d\ast d\vp=&\left(\theta(B^t)(\theta(B)\omega_7-\theta(A)\omega_1)+\theta(C^t)(\theta(C)\omega_7-\theta(A)\omega_2)\right)\wedge e^{7} \label{sdsdphi} \\
           &\left(\theta(B^t)(\theta(B)\omega_2-\theta(C)\omega_1)-\theta(A^t)(\theta(C)\omega_7-\theta(A)\omega_2)\right)\wedge e^{2} \notag \\
           &\left(-\theta(C^t)(\theta(B)\omega_2-\theta(C)\omega_1)-\theta(A^t)(\theta(B)\omega_7-\theta(A)\omega_1)\right)\wedge e^{1}, \notag \\
d\psi =& \left(\theta(A)\omega_7+\theta(B)\omega_1+\theta(C)\omega_2\right)\wedge e^{127}, \label{dsphi} \\
\ast d \psi =& -\left(\theta(A^t)\omega_7+\theta(B^t)\omega_1+\theta(C^t)\omega_2\right),  \label{sdsphi} \\ 
d\ast d \psi =& -\left(\theta(A)\theta(A^t)\omega_7+\theta(A)\theta(B^t)\omega_1+\theta(A)\theta(C^t)\omega_2\right)\wedge e^7 \label{dsdsphi} \\ 
& -\left(\theta(B)\theta(A^t)\omega_7+\theta(B)\theta(B^t)\omega_1+\theta(B)\theta(C^t)\omega_2\right)\wedge e^1\notag\\ 
& -\left(\theta(C)\theta(A^t)\omega_7+\theta(C)\theta(B^t)\omega_1+\theta(C)\theta(C^t)\omega_2\right)\wedge e^2.\notag
\end{align}
\end{proposition}

\begin{remark}
These formulas are valid for any triple of commuting traceless matrices $A,B,C$ beyond compatibilty, which give rise to $G_2$-structures on many other Lie groups $G_{A,B,C}$ not isomorphic to $G_J$.  
\end{remark}

It follows from \cite[(25)]{solvsolitons} that the Ricci operator $\Ricci$ of $(G_{A,B,C},\ip)$ is given by $\la\Ricci\ag,\ngo\ra=0$, $\Ricci|_\ngo=\unm[A,A^t]+\unm[B,B^t]+\unm[C,C^t]$ and 
\begin{equation}\label{Ric}
\Ricci|_\ag=-\left[\begin{smallmatrix}
\tr{S(A)^2} & \tr{S(A)S(B)} & \tr{S(A)S(C)} \\ 
\tr{S(A)S(B)} & \tr{S(B)^2} & \tr{S(B)S(C)} \\ 
\tr{S(A)S(C)} & \tr{S(B)S(C)} & \tr{S(C)^2}
\end{smallmatrix}\right].    
\end{equation}
In particular, $\Ricci\leq 0$ as soon as the three matrices are normal, and by \cite[Theorem 4.8]{solvsolitons}, in the compatible case, $(G_{A,B,C},\ip)$ is a solvsoliton (or expanding Ricci soliton) if and only if $A,B,C$ are all normal matrices and $\Ricci|_\ag=cI$ for some $c\in\RR$.

\subsection{Diagonal case} 
The subgroup $\SO(4)\subset G_2$ mentioned in Section \ref{equiv0} can be used to show the following.  

\begin{lemma}\label{symm}
If $A,B,C$ are all symmetric, then $(G_{A,B,C},\vp)$ is equivalent to $(G_{A_1,B_1,C_1},\vp)$ for some diagonal matrices $A_1,B_1,C_1\in\slg_4(\RR)$.  
\end{lemma}

\begin{proof}
There exists $h\in G_2$ such that if we set $h_2:=h|_{\ngo}\in\SO(4)$, then $A_1=h_2Ah_2^{-1}$, $B_1=h_2Bh_2^{-1}$ and $C_1=h_2Ch_2^{-1}$ are all diagonal matrices.  Thus the isomorphism $G_{A,B,C}\rightarrow G_{A_2,B_2,C_2}$ defined by $h$, where $A_2,B_2,C_2$ are the linear combinations of $A_1,B_1,C_1$ determined by $(h|_\ag)^{-1}$ as in \eqref{equiv1}, determines an equivariant equivalence between the corresponding left-invariant $G_2$-structures, concluding the proof.  
\end{proof}

\begin{remark}\label{equiv2}
According to \eqref{equiv1}, each equivalence class in 
$$
\left\{(G_{A,B,C},\vp): A,B,C \;\mbox{are linearly independent diagonal matrices}\right\}  
$$
has only finitely many elements.  Indeed, $h_2$ must belong to 
$$
N_{\SO(4)}(\ag):=\{ h\in\SO(4):h\ag h^{-1}\subset\ag\}\subset S_4\ltimes\ZZ_2^4,
$$ 
where $\ag\subset\slg_4(\RR)$ is the subspace of all diagonal matrices.   
\end{remark}

We next focus on the case when the matrices $A,B,C$ are all diagonal.  We consider the orthogonal basis of $\Lambda^2\ngo^*$ defined by
\begin{equation}\label{beta}
\bca:=\{\overline{\omega}_7,\overline{\omega}_1,\overline{\omega}_2,\omega_7,\omega_1,\omega_2\}, 
\end{equation}
where $\omega_i$ is as in \eqref{phi2} and 
$$
\overline{\omega}_7:=e^{34}-e^{56}, \qquad \overline{\omega}_1:=e^{35}+e^{46},\qquad \overline{\omega}_2:=-e^{36}+e^{45}.  
$$
Note that the norm of every element in $\bca$ is $\sqrt{2}$,  $\ast_4\omega_i=\omega_i$, $\ast_4\overline{\omega}_i=-\overline{\omega}_i$ for all $i$ and 
$$
\omega_i\wedge\omega_j =\omega_i\wedge\overline{\omega}_j=\overline{\omega}_i\wedge\overline{\omega}_j=0, \qquad\forall i\ne j. 
$$  

The following formulas follow in a straightforward way from Proposition \ref{formulas} and the fact that with respect to the ordered bases $\{ e_3,\dots,e_6\}$ and $\bca$ of $\ngo$ and $\Lambda^2\ngo^*$, respectively,   
\begin{equation}\label{tita-d}
\theta\left(\Diag(x_1,x_2,x_3,x_4)\right) = 
-\left[\begin{smallmatrix}
&&&x_1+x_2&&\\
&0&&&x_1+x_3&\\
&&&&&x_1+x_4\\
x_1+x_2&&&&&\\
&x_1+x_3&&&0&\\
&&x_1+x_4&&&\\
\end{smallmatrix}\right], 
\end{equation}
provided that $x_1+\dots+x_4=0$.    

\begin{proposition}\label{formulas-d}
For each $G_2$-structure $(G_{A,B,C},\vp)$ with $A,B,C\in\slg_4(\RR)$ diagonal matrices, say  
$$
A:=\Diag(a_1,a_2,a_3,a_4), \quad
B:=\Diag(b_1,b_2,b_3,b_4), \quad
C:=\Diag(c_1,c_2,c_3,c_4),   
$$
for some $a_i,b_i,c_i\in\RR$, one has that
\begin{align}
d\vp =& (-(b_1+b_2)\overline{\omega}_7+(a_1+a_3)\overline{\omega}_1)\wedge e^{17} + (-(c_1+c_2)\overline{\omega}_7+(a_1+a_4)\overline{\omega}_2)\wedge e^{27}  \label{dphi-d} \\ 
&+ (-(b_1+b_4)\overline{\omega}_2+(c_1+c_3)\overline{\omega}_1)\wedge e^{12}, \notag\\ 
\ast d\vp=& (-(b_1+b_2)\overline{\omega}_7+(a_1+a_3)\overline{\omega}_1)\wedge e^{2} - (-(c_1+c_2)\overline{\omega}_7+(a_1+a_4)\overline{\omega}_2)\wedge e^{1}  \label{sdphi-d} \\ 
&- (-(b_1+b_4)\overline{\omega}_2+(c_1+c_3)\overline{\omega}_1)\wedge e^{7}, \notag\\ 
\ast d\ast d\vp=& \big(((b_1+b_2)^2+(c_1+c_2)^2)\omega_7-(b_1+b_3)(a_1+a_3)\omega_1  \label{sdsdphi-d}\\ 
&  -(c_1+c_4)(a_1+a_4)\omega_2\big)\wedge e^{7} \notag \\
&+\big(((b_1+b_4)^2+(a_1+a_4)^2)\omega_2-(b_1+b_3)(c_1+c_3)\omega_1 \notag\\ 
& -(a_1+a_2)(c_1+c_2)\omega_7\big)\wedge e^{2} \notag\\
&+\big(-(c_1+c_4)(b_1+b_4)\omega_2+((c_1+c_3)^2+(a_1+a_3)^2)\omega_1 \notag\\ 
& -(a_1+a_2)(b_1+b_2)\omega_7\big)\wedge e^{1}, \notag\\ 
d\psi =& -\big((a_1+a_2)\overline{\omega}_7+(b_1+b_3)\overline{\omega}_1+(c_1+c_4)\overline{\omega}_2\big)\wedge e^{127}, \label{dsphi-d} \\
\ast d \psi =& (a_1+a_2)\overline{\omega}_7+(b_1+b_3)\overline{\omega}_1+(c_1+c_4)\overline{\omega}_2, \label{sdsphi-d} \\ 
d\ast d \psi =& -\left((a_1+a_2)^2\omega_7+(a_1+a_3)(b_1+b_3)\omega_1+(a_1+a_4)(c_1+c_4)\omega_2\right) \wedge e^7 \label{dsdsphi-d} \\ 
& -\left((a_1+a_2)(b_1+b_2)\omega_7+(b_1+b_3)^2\omega_1+(b_1+b_4)(c_1+c_4)\omega_2\right)\wedge e^1 \notag\\ 
& -\left((a_1+a_2)(c_1+c_2)\omega_7+(c_1+c_3)(b_1+b_3)\omega_1+(c_1+c_4)^2\omega_2\right) \wedge e^2.\notag
\end{align}
\end{proposition}

We deduce from \eqref{torsion2}, \eqref{phi2}, \eqref{dphi-d} and \eqref{sdphi-d} the following (recall that $\overline{\omega}_i\wedge\omega_j=0$ for all $i,j$).  

\begin{corollary}\label{formulas-d-cor1}
Any $G_2$-structure $(G_{A,B,C},\vp)$ with $A,B,C\in\slg_4(\RR)$ diagonal matrices satisfies that, 
$$
\tau_0=0, \qquad \tau_1=0, \qquad \tau_2=-\ast d\psi, \qquad \tau_3=\ast d\vp. 
$$ 
\end{corollary}

\begin{remark}
All these $G_2$-structures therefore belong to the class $\wca_2\oplus\wca_3$, which has been characterized in \cite{Agr} as those $G_2$-structures whose corresponding unit spinor is harmonic.  
\end{remark}

\begin{remark}\label{equiv3}
It follows from Remark \ref{equiv2} that each of these $G_2$-structures $(G_{A,B,C},\vp)$ with $A,B,C$ diagonal and linearly independent belongs to a $9$-parameter open subfamily which is pairwise non-equivalent.  
\end{remark}

We can also use Proposition \ref{formulas-d} to compute the Laplacians
$$
\Delta\vp=\ast d\ast d\vp - d\ast d\psi \quad\mbox{and}\quad \Delta\psi=-\ast d\ast d\psi + d\ast d\ast\psi =\ast\Delta\vp.  
$$

\begin{corollary}\label{formulas-d-cor2}
For any $G_2$-structure $(G_{A,B,C},\vp)$, where $A,B,C\in\slg_4(\RR)$ are diagonal matrices, one has that 
\begin{align}
\Delta\vp = & \left((a_1+a_2)^2+(b_1+b_2)^2+(c_1+c_2)^2\right)\omega_7\wedge e^7 \notag\\ 
& +\left((a_1+a_3)^2+(b_1+b_3)^2+(c_1+c_3)^2\right)\omega_1 \wedge e^{1} \label{Deltaphi}\\ 
&  +\left((a_1+a_4)^2+(b_1+b_4)^2+(c_1+c_4)^2\right)\omega_2\wedge e^{2}. \notag \\
\Delta\psi = & \left((a_1+a_2)^2+(b_1+b_2)^2+(c_1+c_2)^2\right)\omega_7\wedge e^{12} \notag\\ 
& +\left((a_1+a_3)^2+(b_1+b_3)^2+(c_1+c_3)^2\right)\omega_1 \wedge e^{27} \label{Deltasphi}\\ 
&  -\left((a_1+a_4)^2+(b_1+b_4)^2+(c_1+c_4)^2\right)\omega_2\wedge e^{17}. \notag
\end{align}
\end{corollary}

\subsection{Closed $G_2$-structures}
The following $3$-parameter family of {\it closed} $G_2$-structures (i.e., $d\vp=0$) was found and studied in \cite{LS-ERP}.

\begin{example}\label{abc}
Consider $(G_{A_0,B_0,C_0},\vp)$ for 
$$
A_0:=\Diag(a,a,-a,-a), \quad
B_0:=\Diag(b,-b,b,-b), \quad
C_0:=\Diag(c,-c,-c,c),   
$$
where $a,b,c\in\RR$, which we will denote by $(G_{a,b,c},\vp)$ from now on.  It follows from \eqref{dphi-d} that  $(G_{a,b,c},\vp)$ is always closed and it was proved in \cite{LS-ERP} that the following conditions are equivalent provided that $a\geq b\geq c> 0$ (in particular $G_{a,b,c}$ is isomorphic to $G_J$) and $a^2+b^2+c^2=3$:

\begin{enumerate}[{\small $\bullet$}]
\item $(G_{a,b,c},\vp)$ is extremally Ricci pinched. 
\item $(G_{a,b,c},\vp)$ is a (steady) Laplacian soliton, i.e., $d\tau_2=\lca_X\vp$ for some vector field $X$.   
\item $(G_{a,b,c},\ip)$ is an (expanding) Ricci soliton. 
\item $(G_{a,b,c},\vp)$ is quadratic (i.e., $d\tau = \frac{1}{7}(1+q)|\tau|^2\vp + q\ast(\tau\wedge\tau)$ for some $q\in\RR$). 
\item $F(a,b,c):=F((G_{a,b,c},\ip))=3$ (see \eqref{RP}).  It follows from \eqref{Ric} that $F(a,b,c)\leq 3$ for all $a,b,c$.  
\item $a=b=c=1$.
\end{enumerate}
Furthermore, all the Laplacian flow solutions $\vp(t)$ (i.e., $\dpar\vp(t)=\Delta\vp(t)$) starting at the closed $G_2$-structure on $G_J$ determined by any element of this family are immortal and satisfy that $(24)^{-3/2}|\tau_2(t)|_t^3\vp(t)$ smoothly converges up to pull-back by diffeomorphisms to the Laplacian soliton on $G_J$ defined by $(G_{1,1,1},\vp)$ as $t\to\infty$.
\end{example}

The diagonal closed case is exhausted by the above example.  

\begin{proposition}\label{diag-c}
If $A,B,C$ are all diagonal, then $(G_{A,B,C},\vp)$ has $d\vp=0$ if and only if it is equal to $(G_{a,b,c},\vp)$ for some $a,b,c\in\RR$ (see Example \ref{abc}).   In that case, $\Delta\vp=d\tau_2$ and
\begin{align*}
\tau_2 =& -2a\overline{\omega}_7-2b\overline{\omega}_1-2c\overline{\omega}_2, \\ 
\Delta\vp =& d\tau_2 = 4a^2\omega_7\wedge e^7+4b^2\omega_1\wedge e^1+4c^2\omega_2\wedge e^2.  
\end{align*}
\end{proposition}

\begin{proof}
The first statement easily follows from \eqref{dphi-d} and \eqref{tita-d} and the formulas for $\tau_2$ and $d\tau_2$ from \eqref{sdsphi-d} and \eqref{dsdsphi-d}, respectively.    
\end{proof}

In the general case, it follows from \eqref{dphi} that $(G_{A,B,C},\vp)$ is closed if and only if 
$$
\theta(B)\omega_7=\theta(A)\omega_1, \qquad  \theta(C)\omega_7=\theta(A)\omega_2, \qquad \theta(B)\omega_2=\theta(C)\omega_1.  
$$
It is easy to see that this holds if and only if the matrices of $\theta(A),\theta(B),\theta(C)$ with respect to the basis $\bca$ given in \eqref{beta} have respectively the form 
$$
\left[\begin{smallmatrix}
&&&a_{14}&a_{15}&a_{16}\\
&A_1&&a_{24}&a_{25}&a_{26}\\
&&&a_{34}&a_{35}&a_{36}\\
a_{14}&a_{24}&a_{34}&&&\\
a_{15}&a_{25}&a_{35}&&0&\\
a_{16}&a_{26}&a_{36}&&&\\
\end{smallmatrix}\right], \quad 
\left[\begin{smallmatrix}
&&&a_{15}&b_{15}&b_{16}\\
&B_1&&a_{25}&b_{25}&b_{26}\\
&&&a_{35}&b_{35}&b_{36}\\
a_{15}&a_{25}&a_{35}&&&\\
b_{15}&b_{25}&b_{35}&&0&\\
b_{16}&b_{26}&b_{36}&&&\\
\end{smallmatrix}\right], \quad 
\left[\begin{smallmatrix}
&&&a_{16}&b_{16}&c_{16}\\
&C_1&&a_{26}&b_{26}&c_{26}\\
&&&a_{36}&b_{36}&c_{36}\\
a_{16}&a_{26}&a_{36}&&&\\
b_{16}&b_{26}&b_{36}&&0&\\
c_{16}&c_{26}&c_{36}&&&\\
\end{smallmatrix}\right],
$$
for some skew-symmetric $3\times 3$ matrices $A_1,B_1,C_1$.  We do not know if $A_1,B_1,C_1$ must all necessarily vanish in the compatible case.  If this was true, then it would follow from Lemma \ref{symm} and Proposition \ref{diag-c} that any closed $(G_{A,B,C},\vp)$ with $G_{A,B,C}$ isomorphic to $G_J$ is equivalent to some $(G_{a,b,c},\vp)$.  Otherwise, these triples may provide new $G_2$-structures that are not equivalent to any $(G_{a,b,c},\vp)$.

\subsection{Coclosed $G_2$-structures}
We now use the formulas in Proposition \ref{formulas-d} to study the {\it coclosed} case (i.e., $d\psi=0$). 

\begin{proposition}\label{diag-cc}
Assume that $A,B,C$ are all diagonal .  Then $(G_{A,B,C},\vp)$ satisfies that $d\psi=0$ if and only if 
$$
A=\Diag(a_1,-a_1,a_2,-a_2), \quad
B=\Diag(b_1,b_2,-b_1,-b_2), \quad
C=\Diag(c_1,c_2,-c_2,-c_1),  
$$
for some $a_i,b_i,c_i\in\RR$.  In that case, $\Delta\vp=d^*d\vp$, $\Delta\psi=d\tau_3=\ast\Delta\vp$ and
\begin{align}
\tau_3 = & \left(-(b_1+b_2)\overline{\omega}_7+(a_1+a_2)\overline{\omega}_1\right)\wedge e^{2} 
 +\left((c_1+c_2)\overline{\omega}_7-(a_1-a_2)\overline{\omega}_2\right)\wedge e^{1}  \label{tau3}\\ 
& + \left((b_1-b_2)\overline{\omega}_2-(c_1-c_2)\overline{\omega}_1\right)\wedge e^{7}, \notag \\ 
\Delta\vp =& \left((b_1+b_2)^2+(c_1+c_2)^2\right)\omega_7\wedge e^7  
 +\left((b_1-b_2)^2+(a_1-a_2)^2\right)\omega_2\wedge e^2 \label{Dphi}\\ 
& +\left((c_1-c_2)^2+(a_1+a_2)^2\right)\omega_1\wedge e^1, \notag \\
\Delta\psi  =& \left((b_1+b_2)^2+(c_1+c_2)^2\right)\omega_7\wedge e^{12}  
 -\left((b_1-b_2)^2+(a_1-a_2)^2\right)\omega_2\wedge e^{17} \label{Dsphi} \\ 
& +\left((c_1-c_2)^2+(a_1+a_2)^2\right)\omega_1\wedge e^{27}. \notag
\end{align}
\end{proposition}

\begin{remark}\label{equiv4}
By Remark \ref{equiv3}, for each of these coclosed $G_2$-structures $(G_{A,B,C},\vp)$ with $A,B,C$ diagonal and linearly independent there is an open neighborhood (depending on $6$ parameters) which is pairwise non-equivalent.  
\end{remark}

\begin{proof}
By using \eqref{dsphi-d} and \eqref{tita-d}, it can be easily shown that the matrices must have that form.  On the other hand, the formulas for $\tau_3$, $\Delta\vp$ and $\Delta\psi=\ast\Delta\vp$ follow from \eqref{sdphi-d} and \eqref{sdsdphi-d}, respectively.  
\end{proof}

Proposition \ref{diag-cc} is therefore providing a family of coclosed $G_2$-structures depending on six parameters $(a_1,\dots, c_2)$.  Recall that one of these $(G_{A,B,C},\vp)$, which will be denoted by 
$$
(G_{(a_1,\dots,c_2)},\vp)
$$ 
from now on,  is identified with a coclosed $G_2$-structure on $G_J$ if and only if $\{ A,B,C\}$ is linearly independent (or compatible).  Note that such condition defines an open and dense subset of $\RR^6$.    

It follows from \eqref{Ric} that the Ricci operator $\Ricci$ of such $(G_{(a_1,\dots,c_2)},\ip)$ is given by $\Ricci|_\ngo=0$, $\Ricci\ag\subset\ag$ and 
\begin{equation}\label{ric-2}
\Ricci|_\ag=-\left[\begin{smallmatrix}
2(a_1^2+a_2^2) & (a_1-a_2)(b_1-b_2) & (a_1+a_2)(c_1-c_2) \\ 
(a_1-a_2)(b_1-b_2) & 2(b_1^2+b_2^2) & (b_1+b_2)(c_1+c_2) \\ 
(a_1+a_2)(c_1-c_2) & (b_1+b_2)(c_1+c_2) & 2(c_1^2+c_2^2) 
\end{smallmatrix}\right].
\end{equation}
In particular, the following conditions are equivalent: 
\begin{enumerate}[{\small $\bullet$}]
\item $(G_{(a_1,\dots,c_2)},\ip)$ is torsion-free (see \eqref{tau3}).  

\item $(G_{(a_1,\dots,c_2)},\ip)$ is Ricci flat (or flat).  

\item $a_1=\dots=c_2=0$.  
\end{enumerate}

\subsection{Laplacian coflow solitons}\label{ccLS}

Natural ways to evolve coclosed $G_2$-structures on a given manifold $M$ are the {\it Laplacian coflow}, defined by 
$$
\dpar\psi(t)=\Delta\psi(t), 
$$
and the {\it modified Laplacian coflow} given by 
$$
\dpar\psi(t)=\Delta\psi(t) + 2 d\big((m-\tr{T})\vp(t)\big), \qquad  m\in\RR, 
$$
where $T$ is the (full) torsion tensor (i.e., $\iota_{T(X)}\psi = \nabla_X\vp$).  We refer to the survey \cite{Grg} and the references therein for more information.  A $G_2$-structure $\psi$ flows self-similarly according to the Laplacian coflow (resp. modified Laplacian coflow) if and only if $\Delta\psi = \lambda\psi + \lca_X\psi$ (resp. $\Delta\psi + 2 d\big((m-\tr{T})\vp\big)= \lambda\psi + \lca_X\psi$) for some $\lambda\in\RR$ and $X\in{\mathfrak X}(M)$, and in that case it is called a {\it (modified) Laplacian coflow soliton}.  We refer to \cite{BggFrnFin, BggFin, KrgMcKTsu, MrnSrp} for examples of (modified) Laplacian coflow solitons in diverse contexts.     

On a simply connected Lie group $G$, left-invariant self-similar solutions to these flows are provided by {\it algebraic solitons} (see \cite{LF, sol-HS}), i.e., 
\begin{equation}\label{algsol}
\Delta\psi = \lambda\psi + \theta(D)\psi, \qquad \lambda\in\RR, \quad D\in\Der(\ggo), \quad D^t=D.  
\end{equation}
One has here that $\theta(D)\psi=-\lca_{X_D}\psi$, where $X_D$ is the vector field on $G$ defined by the one-parameter subgroup of automorphisms of $G$ with derivatives $e^{tD}\in\Aut(\ggo)$.  Note that $\Delta\psi$ must be replaced by $\Delta\psi+ 2 d\big((m-\tr{T})\vp\big)$ in order to get algebraic solitons for the modified Laplacian coflow.  

In what follows, we study the (modified) Laplacian coflow and its solitons on the $6$-parameter family of 
coclosed $G_2$-structures $(G_{(a_1,\dots,c_2)},\vp)$ given in Proposition \ref{diag-cc}.  

\begin{proposition}\label{LS}
If $\{ A,B,C\}$ is linearly independent, then $(G_{(a_1,\dots,c_2)},\vp)$ is an algebraic soliton for the Laplacian coflow if and only if 
$$
(c_1+c_2)^2=(a_1-a_2)^2-4b_1b_2, \qquad (c_1-c_2)^2=(b_1-b_2)^2-4a_1a_2;  
$$
in that case, $\lambda=2\big((a_1-a_2)^2+(b_1-b_2)^2\big)>0$.  Furthermore, $(G_{(a_1,\dots,c_2)},\vp)$ is never an algebraic soliton for the modified Laplacian coflow.  
\end{proposition}

\begin{remark}\label{equiv5}
For each $4$-tuple $(a_1,a_2,b_1,b_2)$, there is at least one and at most four solutions $(c_1,c_2)$ to the above soliton equations.  This therefore provides a pairwise non-homothetic (see Remark \ref{equiv4}) $3$-parameter family of expanding Laplacian coflow solitons on the Lie group $G_J$, whose corresponding Laplacian coflow solutions are therefore defined for $t\in (T,\infty)$ for some $T<0$.  If $\{ A,B,C\}$ is linearly dependent and satisfies the condition in the proposition, then $(G_{(a_1,\dots,c_2)},\vp)$ is still an algebraic soliton but on a different Lie group.  
\end{remark}

\begin{remark}
The corresponding coclosed $G_2$-structures on a compact quotient $M=G_J/\Gamma$ by a lattice $\Gamma$ are not Laplacian coflow solitons on $M$, since the vector field $X_D$ never descends to $M$.  However, the Laplacian coflow solutions do descend to $M$ to become immortal `locally self-similar' solutions in a sense.   
\end{remark}

\begin{proof}
We first note that the space of symmetric derivations of the Lie algebra of $G_{(a_1,\dots,c_2)}$ is given by 
$$
D=\Diag(0,0,0,d_3,d_4,d_5,d_6), \qquad  d_i\in\RR.  
$$
It follows from \eqref{Dsphi} that the algebraic soliton equation $\Delta\psi = \lambda\psi + \theta(D)\psi$ is equivalent to 
\begin{equation}\label{LS-eq}
\left\{\begin{array}{l}
r\omega_7=\lambda\omega_7 + (d_3+d_4)e^{34}+(d_5+d_6)e^{56}, \\ 
s\omega_2=\lambda\omega_2 + (d_3+d_6)e^{36}+(d_4+d_5)e^{45}, \\ 
t\omega_1=\lambda\omega_1 + (d_3+d_5)e^{35}-(d_4+d_6)e^{46}, \\ 
\lambda+d_3+\dots+d_6=0,
\end{array}\right.
\end{equation}
for certain positive numbers $r,s,t$.  The solutions to this system are precisely 
$$
d_3=\dots=d_6=d, \qquad r=s=t=\lambda+2d, \qquad \lambda=-4d,
$$
which implies that $\lambda=2r>0$.  The equations in the proposition correspond to $r=s=t$.   

On the other hand, the algebraic soliton equation $\Delta\psi + 2 d\big((m-\tr{T})\vp\big) = \lambda\psi + \theta(D)\psi$ for the modified Laplacian coflow adds an extra term to each of the three first equations in \eqref{LS-eq} given by certain linear combinations of $\overline{\omega}_i$'s (see \eqref{dphi-d}).  Thus such three linear combinations must vanish if $m-\tr{T}=m-\frac{7}{4}\tau_0=m$ is nonzero, yielding to a contradiction. 
\end{proof}

According to \eqref{ric-2} and \cite[Theorem 4.8]{solvsolitons}, $(G_{(a_1,\dots,c_2)},\ip)$ is a solvsoliton (i.e., $\Ricci=\lambda I+D$ for some $\lambda\in\RR$ and $D\in\Der(\ggo)$) if and only if either
\begin{equation}\label{LS1}
a_2=a_1, \; b_2=-b_1, \; c_2=c_1, \quad a_1^2=b_1^2=c_1^2,
\end{equation}
or,
\begin{equation}\label{LS2}
a_2=-a_1, \; b_2=b_1, \; c_2=-c_1, \quad a_1^2=b_1^2=c_1^2.   
\end{equation}
Note that in both cases, $(G_{(a_1,\dots,c_2)},\vp)$ is also a Laplacian coflow soliton with $\lambda=8a_1^2$ by Proposition \ref{LS}. 

The following example is in some sense the coclosed counterpart of the family of closed $G_2$-structures $(G_{a,b,c},\vp)$ studied in \cite{LS-ERP} (see Example \ref{abc}).   

\begin{example} \label{a1b1c1}
Consider the coclosed $G_2$-structures $(G_{(a_1,\dots,c_2)},\vp)$ such that  
$$
a_2=a_1, \quad b_2=-b_1, \quad c_2=c_1,   
$$
which we denote by $(G_{(a_1,b_1,c_1)},\vp)$.  According to Proposition \ref{diag-cc}, 
\begin{align*}
\tau_3 = & 2a_1\overline{\omega}_1\wedge e^{2} 
 +2c_1\overline{\omega}_7\wedge e^{1} + 2b_1\overline{\omega}_2\wedge e^{7}, \\ 
\Delta\vp =& 4c_1^2\omega_7\wedge e^7  + 4b_1^2\omega_2\wedge e^2 + 4a_1^2\omega_1\wedge e^1, \\
\Delta\psi  =& 4c_1^2\omega_7\wedge e^{12}  - 4b_1^2\omega_2\wedge e^{17} + 4a_1^2\omega_1\wedge e^{27}. 
\end{align*}
On the other hand, the Ricci operator is given by 
\begin{equation}\label{ric-a1}
\Ricci=\Diag(-4a_1^2,-4b_1^2,-4c_1^2,0,0,0,0),
\end{equation}
with respect to the basis $\{ e_7, e_1,\dots, e_6\}$ (see \eqref{ric-2}), showing that the family  
$$
\left\{(G_{(a_1,b_1,1)},\vp):a_1\geq b_1\geq 1\right\}
$$ 
is pairwise non-homothetic as the corresponding metrics are pairwise non-homothetic.  
The following conditions are therefore equivalent: 
\begin{enumerate}[{\small $\bullet$}]
\item $(G_{(a_1,b_1,c_1)},\vp)$ is a Laplacian coflow soliton.  

\item $(G_{(a_1,b_1,c_1)},\ip)$ is a solvsoliton.  

\item $F(a_1,b_1,c_1)= 3$ (see \eqref{RP}).  

\item $a_1^2=b_1^2=c_1^2$ (cf.\ \eqref{LS1}).  
\end{enumerate}
It follows from \eqref{ric-a1} that $F(a_1,b_1,c_1)\leq 3$, that is, the Laplacian coflow solitons are the only global maxima of $F$.   
\end{example}

\begin{example}\label{curve1} {\it Curve of solitons containing \eqref{LS1}}.  
According to Proposition \ref{LS}, the coclosed $G_2$-structure 
$$
(a_1,a_2,1,-1,c_1,c_2), \qquad a_1^2+a_2^2=2, 
$$
is a Laplacian coflow soliton provided that 
$$
(c_1+c_2)^2=2(3-a), \qquad (c_1-c_2)^2=4(1-a), \qquad a:=a_1a_2.  
$$
It follows from \eqref{ric-2} that,
$$
F(a)=\frac{9(3-a)^2}{5a^2-34a+41}, \quad\mbox{and so} \quad F'(a)=\frac{36(3-a)(a+5)}{(5a^2-34a+41)^2}.    
$$
Thus $F$ is strictly increasing for $-1\leq a\leq 1$, which corresponds to $a_1\in[1,\sqrt{2}]$ and $a_2=\pm\sqrt{2-a_1^2}\in [-1,1]$.  We therefore obtain that this gives an explicit pairwise non-homothetic family of Laplacian coflow solitons on $G_J$, which coincides with \eqref{LS1} at $a=1$ (i.e., $a_1=a_2=1$).  Note that $1.8 = F(-1) \leq F(a) \leq F(1)=3$ for any $a\in [-1,1]$.    
\end{example}

\subsection{Laplacian coflow evolution}
We study in this section the behavior of the Laplacian coflow among the $6$-parameter family of cocolosed $G_2$-structures $(G_{(a_1,\dots,c_2)},\vp)$ given in Proposition \ref{diag-cc} by using the bracket flow approach (see \cite[Section 3.3]{LF} or \cite{sol-HS} for more information).  

Given the Lie bracket of $\ggo_{(a_1,\dots,c_2)}$, 
$
\mu=\mu(a_1,a_2,b_1,b_2,c_1,c_2),
$
we set
$$
r:=(b_1+b_2)^2+(c_1+c_2)^2, \quad s:=(b_1-b_2)^2+(a_1-a_2)^2, \quad t:=(c_1-c_2)^2+(a_1+a_2)^2.  
$$
It follows from \eqref{Dsphi} that $\Delta\psi=\theta(Q_\mu)\psi$, where
$$
Q_\mu=\unm\Diag(-r+s+t, r+s-t, r-s+t, 0,0,0,0),
$$ 
and so a straightforward computation gives that the bracket flow 
$$
\mu'=\theta(Q_\mu)\mu:= Q_\mu\mu(\cdot,\cdot)-\mu(Q_\mu\cdot,\cdot)-\mu(\cdot,Q_\mu\cdot),
$$ 
is equivalent to the following ODE system:
$$
\left\{
\begin{array}{lcl}
a_1'=-\unm(-r+s+t)a_1,  &\quad&  a_2'=-\unm(-r+s+t)a_2,  \\
b_1'=-\unm(r+s-t)b_1,  &\quad&  b_2'=-\unm(r+s-t)b_2,  \\ 
c_1'=-\unm(r-s+t)c_1, &\quad&  c_2'=-\unm(r-s+t) c_2,  
\end{array}\right.
$$
which is in turn given by
\begin{equation}\label{ode}
\left\{
\begin{array}{l}
a_i'=\left(-(a_1^2+a_2^2)+2b_1b_2+2c_1c_2\right)a_i, \qquad i=1,2,   \\
b_i'=\left(-(b_1^2+b_2^2)+2a_1a_2-2c_1c_2\right)b_i, \qquad i=1,2,    \\ 
c_i'=\left(-(c_1^2+c_2^2)-2a_1a_2-2b_1b_2\right)c_i,  \qquad i=1,2.  
\end{array}\right.
\end{equation}

We first observe that it is not possible to prove {\it long-time existence} (i.e., the solution is defined for all $t\in [0,\infty)$, also called {\it immortal} solutions) for all the solutions of \eqref{ode} by showing that the square norm function $N:=a_1^2+\dots+c_2^2$ is always non-increasing, since this is not the case.  Indeed, an easy computation gives that $N'>0$ at any point $(a,a,1,-1,c,-c)$ such that $c>1$ and $a$ is sufficiently small.  

We next show long-time existence among two different subfamilies.  

\begin{example}\label{flow1} {\it Long-time existence I}.  
It follows from \eqref{ode} that the $3$-parameter family of coclosed $G_2$-structures considered in Example \ref{a1b1c1}, 
$$
\left(a,a,b,-b,c,c\right),  
$$
is invariant for the Laplacian coflow and the evolution is given by 
\begin{equation}\label{ode1}
\left\{
\begin{array}{l}
a'=\left(-2a^2-2b^2+2c^2\right)a,    \\
b'=\left(-2b^2+2a^2-2c^2\right)b,    \\
c'=\left(-2c^2-2a^2+2b^2\right)c.   
\end{array}\right.
\end{equation}
It is easy to see that the derivative of the square norm $N:=a^2+b^2+c^2$ satisfies that
\begin{align*}
\unm N' &= -2(a^4+b^4) - 4c^4 + 2c^2(a^2-b^2) + 4c^2(b^2-a^2) \\ 
&= -b^4-c^4+2b^2c^2 +(-2a^2-b^4-3c^4-2a^2c^2)\leq 0.
\end{align*}
Thus $N$ is strictly decreasing unless $a=b=c=0$, the only critical point of the ODE \eqref{ode1}.  In particular, any solution stays in a compact subset and so it is defined for all $t\in[0,\infty)$.  
\end{example}

\begin{example}\label{flow2} {\it Long-time existence II}.  
The $4$-parameter family of coclosed $G_2$-structures 
$$
\left(a,a,b,b,c_1,c_2\right),  
$$
is also invariant for the Laplacian coflow by \eqref{ode}, which is equivalent to
\begin{equation}\label{ode2}
\left\{
\begin{array}{l}
a'=\left(-2a^2+2b^2+2c_1c_2\right)a,    \\
b'=\left(-2b^2+2a^2-2c_1c_2\right)b,    \\
c_i'=\left(-(c_1^2+c_2^2)-2a^2-2b^2\right)c_i,  \qquad i=1,2.    
\end{array}\right.
\end{equation}
The function $N:=a^2+b^2+c_1^2+c_2^2$ therefore has
\begin{align*}
\unm N' 
&\leq -2(a^4+b^4) - (c_1^2+c_2^2)^2 + 4a^2b^2 - 2a^2(c_1^2+c_2^2-c_1c_2)-2b^2(c_1^2+c_2^2+c_1c_2) \\ 
&\leq -2(a^2-b^2)^2 - (c_1^2+c_2^2)^2  \leq 0,  
\end{align*}
and so $N$ is strictly decreasing unless $a^2=b^2$ and $c_1=c_2=0$, which are precisely the critical points of the flow.  Thus all these solutions are immortal and it follows from \cite[Theorem 3.8]{LF} or \cite[Corollary 6.5]{sol-HS}  that $(a,a, a, a,0,0)$ is a steady Laplacian coflow soliton (on a Lie group different from $G_J$).  
\end{example}

The bracket flow can also be used to study the convergence behavior of solutions (see \cite[Corollary 3.6]{LF}).  

\begin{example}\label{conv1} {\it Convergence to solitons I}. 
We analyze here the behavior at infinity of the $3$-parameter family in Example \ref{flow1}.  It was shown in Example \ref{a1b1c1} that $(1,1,1)$ is the only Laplacian coflow soliton among $a,b,c>0$, $a^2+b^2+c^2=3$.  By applying the method of projecting solutions on the Poincar\'e sphere (see, e.g., \cite[Section 3.10, Theorems 4 and 5]{Prk}), it is not hard to see that any positive solution to \eqref{ode1} satisfies that 
$$
\tfrac{\sqrt{3}}{\sqrt{a^2+b^2+c^2}} (a,b,c) \to (1,1,1).  
$$  
It follows from \cite[Corollary 3.6]{LF} that any Laplacian coflow flow solution $\vp(t)$ starting at one of these coclosed $G_2$-structures on $G_J$ satisfies that $u_0|\tau_3(t)|_t^3\vp(t)$ (for some constant $u_0>0$) smoothly converges up to pull-back by diffeomorphisms, as $t\to\infty$,  to the Laplacian coflow soliton on $G_J$ defined by $(1,1,1)$.
\end{example}

\begin{example}\label{conv2} {\it Convergence to solitons II}. 
We finally consider the $3$-parameter family 
$$
\left(a,a,b,b,c,c\right),  
$$
which is a subfamily of Example \ref{flow2}, it is invariant for the Laplacian coflow and the evolution is given by 
\begin{equation}\label{ode3}
\left\{
\begin{array}{l}
a'=2\left(-a^2+b^2+c^2\right)a,    \\
b'=2\left(-b^2+a^2-c^2\right)b,    \\
c'=2\left(-c^2-a^2-b^2\right)c.     
\end{array}\right.
\end{equation}
Concerning convergence, the fact that there are no compatible (i.e., $a,b,c\ne 0$) Laplacian coflow solitons in this family by Proposition \ref{LS} is particularly interesting.  Where are the solutions heading? What one obtains from the Poincar\'e sphere method is that the possible limits of normalized solutions $\tfrac{1}{\sqrt{a^2+b^2+c^2}} (a,b,c)$ in the first octant are given by $\tfrac{1}{\sqrt{2}} (1,1,0)$, $(1,0,0)$ and $(0,1,0)$.  Note that these correspond to steady solitons on Lie groups which are not isomorphic to $G_J$.  By using that $ab$ is constant in time, we obtain that all these normalized solutions are converging to the steady soliton $\tfrac{1}{\sqrt{2}} (1,1,0)$.  It follows from \cite[Corollary 3.6]{LF} that the solution $\vp(t)$ on $G_J$ smoothly converges up to pull-back by diffeomorphisms and scaling, as $t\to\infty$, to the steady Laplacian coflow soliton $(G_{(1,1,0)},\vp)$.
\end{example}


\begin{thebibliography}{MMM}

\bibitem[ACFH]{Agr} {\sc I. Agricola, S. Chiossi, T. Friedrich, J. H\"oll}, Spinorial description of $\SU(3)-$ and $G_2$-manifolds, {\it J. Geom. Phys.} {\bf 98} (2015), 535-555.

\bibitem[A]{Alk} {\sc D. Alekseevskii}, Conjugacy of polar factorizations of Lie groups, {\it Mat. Sb.} {\bf 84} (1971), 14-26; {\it English translation}: {\it Math. USSR-Sb.} {\bf 13} (1971), 12-24.

%

\bibitem[BFF]{BggFrnFin} {\sc L. Bagaglini, M. Fern\'andez, A. Fino}, Laplacian co-flow on the quaternionic Heisenberg group, {\it Asian J. Math.}, in press.  

\bibitem[BF]{BggFin} {\sc L. Bagaglini, A. Fino}, The Laplacian coflow on almost-abelian Lie groups, {\it Ann. Mat. Pura Appl.} {\bf 197} (2018), 1855-1873.


\bibitem[Ba]{Bll} {\sc G. Ball}, Seven-Dimensional Geometries With Special Torsion, Ph.D. dissertation, Duke Univ..

\bibitem[B]{Bry} {\sc R. Bryant}, Some remarks on $G_2$-structures, Proc. G\"okova Geometry-Topology Conference (2005), 75-109.

\bibitem[CI1]{ClyIvn1} {\sc R. Cleyton, S. Ivanov}, On the geometry of closed $G_2$-structures,  {\it Comm. Math. Phys.} {\bf 270} (2007), 53-67.

\bibitem[CI2]{ClyIvn2} {\sc R. Cleyton, S. Ivanov}, Curvature decomposition of G2-manifolds. {\it J. Geom. Phys.} {\bf 58} (2008), 1429-1449.


\bibitem[FFR]{FrnFinRff} {\sc M. Fern\'andez, A. Fino, A. Raffero}, Locally conformal calibrated $G_2$-manifolds,  {\it Ann. Mat. Pura Appl.} {\bf 195} (2016), 1721-1736.

\bibitem[FR]{FinRff2} {\sc A. Fino, A. Raffero}, A class of eternal solutions to the $G_2$-Laplacian flow, preprint 2018 (arXiv).

\bibitem[GW]{GrdWls} {\sc C. S. Gordon, E. N. Wilson}, Isometry groups of Riemannian solvmanifolds, {\it Trans. Amer. Math. Soc.} {\bf 307} (1988), 245–-269.

\bibitem[G]{Grg} {\sc S. Grigorian}, Flows of co-closed $G_2$-structures, {\it Fields Institute Communications}, Springer, in press (arXiv).

\bibitem[KMT]{KrgMcKTsu} {\sc S. Karigiannis, B. McKay, M.-P. Tsui}, Soliton solutions for the Laplacian co-flow of some $G_2$-structures with symmetry, {\it Diff. Geom. Appl.} {\bf 30} (2012), 318-333.

\bibitem[L1]{solvsolitons} {\sc J. Lauret}, Ricci soliton solvmanifolds, {\it J. reine angew. Math.} {\bf 650} (2011), 1-21.

\bibitem[L2]{LF}  {\sc J. Lauret}, Laplacian flow of homogeneous $G_2$-structures and its solitons, {\it Proc. London Math. Soc.} {\bf 114} (2017), 527-560.

\bibitem[L3]{LS-ERP}  {\sc J. Lauret}, Laplacian solitons: Questions and homogeneous examples, {\it Diff. Geom. Appl.} {\bf 54} (2017), 345-360.

\bibitem[L4]{Fields}  {\sc J. Lauret}, Distinguished $G_2$-structures on solvmanifolds, {\it Fields Institute Communications}, Springer, in press (arXiv).

\bibitem[L5]{sol-HS}  {\sc J. Lauret}, The search for solitons on homogeneous spaces, preprint 2019 (arXiv).  

\bibitem[LN1]{ERP}  {\sc J. Lauret, M. Nicolini}, Extremally Ricci pinched $G_2$-structures on Lie groups, {\it Comm. Anal. Geom.}, in press (arXiv).

\bibitem[LN2]{ERP2}  {\sc J. Lauret, M. Nicolini}, The classification of ERP $G_2$-structures on Lie groups, {\it Ann. Mat. Pura App.}, in press (arXiv).

\bibitem[Lo]{Lty} {\sc J. Lotay}, Geometric flows of $G_2$ structures, {\it Fields Institute Communications}, Springer, in press (arXiv).

\bibitem[LoW]{LtyWei} {\sc J. Lotay, Y. Wei}, Laplacian flow for closed $G_2$ structures: Shi-type estimates, uniqueness and compactness, {\it Geom. Funct. Anal.} {\bf 27} (2017), 165-233.

\bibitem[MOV]{MnrOtlVll} {\sc V. Manero, A. Otal, R. Villacampa}, Laplacian coflow for warped $G_2$-structures, {\it Diff. Geom. Appl.}, in press.   

\bibitem[MS]{MrnSrp} {\sc A. Moreno, H. S\'a Earp}, Explicit Soliton for the Laplacian Co-Flow on a Solvmanifold, {\it Sao Paulo J. Math.}, in press. 

\bibitem[N1]{Ncl}  {\sc M. Nicolini}, New examples of shrinking Laplacian solitons, in preparation.    

\bibitem[N2]{Ncl-T} {\sc M. Nicolini}, $G_2$-estructuras ERP en grupos de Lie, Ph.D. dissertation, Univ. Nac. de C\'ordoba, Argentina.  Defense date: June 26th, 2020.  

\bibitem[P]{Prk} {\sc L. Perko}, Differential equations and dynamical systems, {\it TAM} {\bf 7} (third edition, 2001), Springer-Verlag.  

\bibitem[PR]{PdsRff} {\sc F. Podesta, A. Raffero}, On the automorphism group of a closed $G_2$-structure, {\it Quart. J. Math.}, in press.

\bibitem[VM]{VnLMnr} {\sc H. Van Le, M. Munir}, Classification of compact homogeneous spaces with invariant $G_2$-structures, {\it Adv. Geom.} {\bf 12} (2012), 303-328. 
\end{thebibliography}
\end{document}